\documentclass[reqno,11pt]{amsart}
\usepackage[left=1 in, right=1 in,top=1 in, bottom=1 in]{geometry}
\usepackage{amsfonts,amsmath,amsthm,amssymb,mathrsfs}
\newtheorem{theorem}{Theorem}[section]
\newtheorem{definition}[theorem]{Definition}
\newtheorem{lemma}[theorem]{Lemma}
\newtheorem{proposition}[theorem]{Proposition}
\newtheorem{remark}[theorem]{Remark}

\usepackage[pdfstartview=FitH,colorlinks=true]{hyperref}
\hypersetup{urlcolor=red, linkcolor=blue,citecolor=red, CJKbookmarks=true}
\allowdisplaybreaks
\makeatletter
\newcommand*{\rom}[1]{\expandafter\@slowromancap\romannumeral #1@}

\newcommand{\bd}{\mathrm{d}}

\newcommand{\grad}{\nabla_{x}}
\newcommand{\vdot}{v\!\cdot\!\grad}
\newcommand{\Q}{\mathcal{Q}}
\newcommand{\xmus}{ |\mathbf{n} \cdot v| \bd v \bd \sigma_x \bd s }

\makeatother
\usepackage{tikz}

\title[Incoming data]{Global existence of renormalized solutions to Boltzmann equations with incoming boundary condition and non-cutoff kernel }

\author[N. Jiang]{Ning Jiang}
\address[Ning Jiang]{\newline School of Mathematics and Statistics, Wuhan University, Wuhan, 430072, P. R. China}
\email{njiang@whu.edu.cn}

\author[X. Zhang]{Xu Zhang}
\address[Xu Zhang]{\newline School of Mathematics and Statistics, Wuhan University, Wuhan, 430072, P. R. China}
\email{xuzhang889@whu.edu.cn}
\begin{document}
	\begin{abstract}
		We prove the existence of global renormalized solutions to the Boltzmann equation in bounded domain with incoming boundary condition, with non-cutoff collision kernels. Thus we extend the results of  \cite{villani2002noncutoff} for whole spaces or periodic domain to bounded domains endorsed with incoming boundary condition.
	\end{abstract}
\maketitle	
\tableofcontents

\section{Introduction}

The Boltzmann equation (or Maxwell-Boltzmann ) equation is an integro-differentiable equation
\begin{equation}\label{BE}
  \partial_t f + \vdot f = \Q(f,f)\,,
\end{equation}
which models the statistical evolution of a rarefied gas. In equation \eqref{BE}, $f(t,x,v)$ is a non-negative measurable function, which denotes the number density of the gas molecules at time $t \geq 0$, at the position $ x \in \Omega$, with velocity $v \in \mathbb{R}^N$, ($N \geq 2$). Here $\Omega$ denotes the whole space $\mathbb{R}^N$, or a torus $\mathbb{T}^N$, or  a bounded domain in $\mathbb{R}^N$. Furthermore, $\Q(f, f)$ is the collision operator whose structure is described below.

In this work, the Boltzmann equation \eqref{BE} is given an initial data which satisfies some natural physical bounds (bounded mass, momentum, energy and entropy, etc.). More specifically,
\begin{equation}\label{IC-1}
  f|_{t=0} = f_0(x,v)\quad \mbox{in}\quad\! \Omega \times \mathbb{R}^N\,,
\end{equation}
which satisfies
\begin{equation}\label{IC-bounds}
  f_0 \geq 0\quad \mbox{a.e.}\quad \mbox{and}\quad\! \iint_{\Omega\times\mathbb{R}^N}f_0(1+ |x|^2 + |v|^2 + |\log f_0|)\,\bd x\bd v < \infty\,.
\end{equation}

The well-posedness of the Boltzmann equation \eqref{BE} is a fundamental problem in mathematical physics. Besides many results on the smooth solutions which required the initial data $f_0$ is ``small" in some functional spaces, the first global in time solution with ``large" data, i.e. the initial data $f_0$ satisfies \eqref{IC-bounds}: only some finite physical bounds, without any smallness requirements on the size of $f_0$, was proved in the celebrated DiPerna-Lions' theorem \cite{lions1989cauchy} for $\Omega = \mathbb{R}^N$ (with some minor modifications, their proof works also for torus $\mathbb{T}^N$ ).

Since in the natural functional spaces of the number density $f(t,x,v)$, say $L^1\cap L\log L$, the collision term $\Q(f,f)$ in \eqref{BE} is not even locally integrable, which makes weak solutions to the Boltzmann equation can not be defined in the usual sense. In stead, under the Grad's angular cutoff assumption and a mild decay condition on the collision kernel which we will describe in details later, DiPerna and Lions defined the so-called renormalized solutions to the \eqref{BE} and prove that a sequence of renormalized solutions which satisfy only the physically natural a priori bounds converge weakly in $L^1$. From this stability they deduced global existence of renormalized solutions.

In \cite{villani2002noncutoff}, Alexandre and Villani studied the Boltzmann equation without Grad's angular cutoff assumption. They introduced a new renormalized formulation that allows the cross section to be singular in both the angular and the relative velocity variables, which occur in long-range interactions and soft potentials in particular Coulomb interaction. Together with some new estimates, they prove global existence of renormalized solutions with defect measure. Again, Alexandre-Villani's results were proved for the whole space case.

Since the obvious importance in applications and theoretical research, extending the global existence results of DiPerna-Lions \cite{lions1989cauchy} and Alexandre-Villani \cite{villani2002noncutoff} for whole space (and torus) to the domain with boundary is a natural open question since then. A first complete answer in this direction is due to Mischler \cite{mischler2010asens}, who proved global renormalized solution of the Boltzmann equation with cutoff collision kernels in a bounded domain endowed with Maxwell reflection boundary condition, based on some new observations on weak-weak convergence and his previous results on the traces of kinetic equations \cite{mischler2000vlasovtrace, mischler2000-cmp}. Maxwell in \cite{Maxwell} proposed this boundary condition, which stated that the gas molecules back to the domain at the boundary come one part from the specular reflection of the molecules escaping the domain, the other part from those entering the wall, interacting with the molecules in the wall, and re-evaporating back to the domain with the thermal dynamical equilibrium state of the wall. This boundary condition for the Boltzmann equation can be viewed as an analogue of the Robin condition for the macroscopic equations. In fact, based on Mischler's global renormalized solutions constructed in \cite{mischler2010asens}, incompressible Navier-Stokes equations with several boundary conditions can be justified \cite{M-S, SRM-LNM-2009, Jiang-Masmoudi-CPAM}. In a forthcoming paper, we extend Mischler's result to the non-cutoff collision kernel case.

Another boundary condition for the kinetic equation is more direct: the number density of the gas molecules back to the domain is prescribed. This is the so-called incoming boundary condition, which has been widely used in applied fields. For more asymptotic analysis, including how to derive the boundary conditions for the incompressible Navier-Stokes equations from the Boltzmann equation with incoming data, see Sone's books \cite{Sone, Sone2}. To prove global renormalized solutions to the Boltzmann equations with incoming boundary condition, for both the cutoff and non-cutoff collision kernels, is the main concern of the current paper. We first introduce more detailed information on the Boltzmann equation in particular the collision kernels and the boundary conditions so that we can state our main results precisely.

\subsection{Collision kernel of the Boltzmann equation}

In the Boltzmann equation \eqref{BE}, $\Q$ is the Boltzmann collision operator, which acts only on the velocity dependence of $f$ quadratically:
\begin{equation}\label{collision-original}
\begin{split}
\Q(f,f)=&\int_{\mathbb{R}^N\times\mathbb{S}^{N-1}}(f'f'_*-ff_*) b(v-v_*,\omega) \bd v_*\bd\omega\,,
\end{split}
\end{equation}
where $f'=f(v')$, $f'_*= f(v'_*)$, $f_*= f(v_*)$ ($t$ and $x$ are only parameters), and the formulae
\begin{align*}
\begin{cases}
v' = \frac{v + v_*}{2} + \frac{|v-v_*|}{2}\omega \\
v_*'= \frac{v + v_*}{2} - \frac{|v-v_*|}{2}\omega\,,
\end{cases}
\end{align*}
yield a parametrization of the set of solutions to the conservation laws of elastic collision
\begin{align*}
\begin{cases}
 v+ v_* = v' + v'_* \\
 |v|^2 + |v_*|^2 =    |v'|^2 + |v'_*|^2\,.
\end{cases}
\end{align*}
Here $v$ and $v_*$ denote the velocities of two particle before the elastic collision, and $v'$ and $v_*'$ denotes the post-collision velocities.  The nonnegative and a.e. finite weight function $b(v-v_*, \omega)$, called {\it cross-section}, is assumed to depend only on the relative velocity $|v-v_*|$ and cosine of the derivation angle  $ (\frac{v-v_*}{|v-v_*|}, \omega)$. For a given interaction model, the cross section can be computed in a semi-explicit way by solving a classical scattering problem, see for instance, \cite{Cercignani88}. A typical example is that in dimension 3, for the inverse $s\mbox{-}$power repulsive forces (where $s > 1$ is the exponent of the potential), if denoted by $\kappa = \frac{v -  v_*}{| v - v_*|}$ and $\omega = \frac{v' -  v_*'}{| v' - v_*'|}$,
\begin{equation}\label{s-power-1}
b(v-v_*, \omega) = |v-v_*|^\gamma b(\kappa \cdot \omega) = |v-v_*|^\gamma b(\cos \theta)\,,\quad \gamma = \tfrac{s-5}{s-1}\,,
\end{equation}
and
\begin{equation}\label{s-power-2}
 \sin\theta b(\cos\theta) \approx K \theta^{-1-s'}\quad\text{as}\quad\!\theta \to 0\,,\quad\text{where}\quad\! s'=\tfrac{2}{s-1}\quad\!\text{and}\quad\! K >0\,.
\end{equation}
Notice that, in this particular situation, $b(z,\omega)$ is not locally integrable, which is not due to the specific form of inverse power potential. In fact, one can show (see \cite{Villani-review}) that a non-integrable singularity arises if and only forces of infinite range are present in the gas. Thus, some assumptions must be made on the cross section to make the mathematical treatment of the Boltzmann equation convenient.

There are basically two types of assumptions on the cross section. The main assumption by DiPerna and Lions in \cite{lions1989cauchy} on the cross section was {\em Grad's angular cutoff}, namely, that {the cross section be integrable}, locally in all variables. More precisely, they assumed
\begin{equation}\label{Grad-cutoff-1}
  A(z)= \int_{S^{N-1}}b(z,\omega)\,\bd \omega \in L^1_{loc}(\mathbb{R}^N)\,,
\end{equation}
together with a condition of mild growth of $A$:
\begin{equation}\label{Grad-cutoff-2}
  (1+|v|^2)^{-1}\int_{|z-v|\leq R}A(z)\,\bd z\to 0\quad\text{as}\quad\! |v|\to \infty\,,\quad\text{for all}\quad\! R<\infty\,.
\end{equation}

However, although the Grad's angular cutoff assumption \eqref{Grad-cutoff-1} has been widely used in this field, it is not satisfactory from the physical point of view. Indeed, as soon as one consider {\em long range} interactions, even with a very fast decay at infinity, this assumption is not satisfied. A typical example is the that of inverse $s\mbox{-}$power repulsive forces in dimension 3 mentioned before. The function $\sin\theta b(\cos\theta)$ in \eqref{s-power-2} presents a {\em non-integrable} singularity as $\theta\to 0$. This regime corresponds to {\em grazing collisions}, i.e. collisions in which particles are hardly deviated. Physically speaking, these are the collisions between particles that are microscopically very far apart, with a large impact parameter. Another complication arises when dealing with the Coulomb potential: For $s=2$ in dimension $N=3$ as in \eqref{s-power-1}, one finds a cross-section behaving like $|v-v_*|^{-3}$ in the relative velocity variable, hence {\em not locally integrable} as a function of the relative velocity (this is called kinetic singularity). The DiPerna-Lions formulation can not handle this case, which is one of the most important from a physical point of view.

In \cite{villani2002noncutoff}, Alexandre and Villani employed several new tools to treat both angular and kinetic singularities and extended the DiPerna-Lions theory to very general, physically realistic long-range interactions, including the Coulomb potential as a limit case. For the readers' convenience, we list below the non-cutoff assumptions made in \cite{villani2002noncutoff} on the cross-section:

\begin{enumerate}

\item  Borderline singularity assumption.    Assume that the cross section has the following decomposition:
\begin{align}\label{AV-first}
b(z,\omega) = \frac{\beta_0(\kappa\cdot \omega)}{|z|^3} + B_1(z, \omega),~~ \kappa = \frac{z}{|z|},
\end{align}
for some nonnegative measurable functions $\beta_0$ and $B_1$, and define
\begin{align}
\mu_0 = \int_{S^2} \beta_0(\kappa \cdot \omega)(1- \kappa\cdot \omega) \,\bd\omega\,,\\
M_1(|z|) = \int_{S^2} B_1(z, \omega)(1- \kappa\cdot \omega)\,\bd\omega\,,\\
M_1'(|z|) = \int_{S^2} B_1'(z, \omega)(1- \kappa\cdot \omega)\,\bd\omega\,,
\end{align}
where
\[ B_1'(z,\omega) = \sup\limits_{1 < \lambda \le \sqrt{2}} \frac{B_1(\lambda z, \omega ) -B_1(z, \omega)}{(\lambda - 1) |z|}. \]
We require that
\[ \mu_0 < +\infty\quad \text{and}\quad  M_1(|z|),~|z|M_1'(|z|) \in L^1_{loc}(\mathbb{R}^3).  \]

\item Behavior at infinity assumption.  For $0\le \alpha \le 2$, let
\begin{align}
M^\alpha(|z|) = \int_{S^2} b(z, \omega)(1- \kappa\cdot \omega)^\frac{\alpha}{2}\,\bd\omega,~~ \kappa = \frac{z}{|z|}.
\end{align}
We require that for some $\alpha \in [0,2]$, as $|z| \to \infty$,
\begin{align}
M^\alpha(|z|) = o(|z|^{2-\alpha})\,,\quad\text{and}\quad\!|z|M'(|z|) = o(|z|^2).
\end{align}

\item  Angular singularity assumption.
\begin{align}
B(z,\omega) \ge \Phi_0(|z|)b_0(\kappa\cdot \omega),~~\kappa = \frac{z}{|z|},
\end{align}
where $\Phi_0$ is a continuous function,~ $\Phi_0(|z|)>0$ if $|z| \neq 0$, and
\begin{align}\label{AV-last}
\int_{S^2} b_0(\kappa \cdot \omega) = \infty.
\end{align}
\end{enumerate}

For the inverse $s\mbox{-}$power repulsive forces in dimension 3, the above three assumptions together allow the following range of parameters:
 \[ \gamma \ge -3,~~ 0 \le s' < 2,~~ s' + \gamma < 2.  \]

Note that when $s=2$, $\gamma= -3$, which corresponds to Coulomb interaction. However, the limiting case $s=2$ is not suited for Boltzmann equation as the Boltzmann collision operator should be replaced by the Landau operator in
order to handle that situation (see \cite{Villani-review}). We will consider the boundary problem for Landau equation in a separate paper.

\subsection{Boundary conditions}

As mentioned before, the main concern of the current paper is to extend   Alexandre-Villani \cite{villani2002noncutoff} theories for  non-cutoff cross section  to
the bounded domain with incoming boundary condition.

Let $\Omega$ be an open and bounded subset of $\mathbb{R}^N$ and set $\mathcal{O}=\Omega\times\mathbb{R}^3$ and $\mathcal{O}_T=(0,T)\times\Omega\times\mathbb{R}^3$. We assume that the boundary $\partial\Omega$ is sufficiently smooth.
The regularity that we need is that there exists a vector field $\mathrm{n}\in W^{2,\infty}(\Omega\,;\mathbb{R}^N)$ such that $\mathrm{n}(x)$ coincides with the outward unit normal vector at $x\in\partial\Omega$. We define $\Sigma^x_{\pm}:= \{v\in \mathbb{R}^N\,; \pm v\cdot\mathrm{n}(x) > 0\}$ the sets of outgoing ($\Sigma^x_+$) and incoming ($\Sigma^x_-$) velocities at point $x\in \partial \Omega$ as well as $\Sigma = \partial\Omega \times \mathbb{R}^N$ and
\begin{equation}\nonumber
\Sigma_{\pm}=\{(x,v)\in\Sigma:\pm v \!\cdot\!\mathrm{n}(x) >0\} = \{(x,v)\,; x\in \partial \Omega\,, v\in \Sigma^x_{\pm}\}\,.
\end{equation}
We also denote by  $\mathrm{d}\sigma_x $ the Lebesgue measure on  ${\partial\Omega}$.

The boundary condition considered in this paper is that the number density on the incoming to the domain is prescribed. More precisely, denoted by
$\gamma f$ be the trace of the number density (provided the trace can be defined), and let $\gamma_{\pm}f = \mathbf{1}_{(0,\infty)\times\Sigma_{\pm}}\gamma f$. The so-called incoming boundary condition is that
\begin{align}\label{boundary-conditions-incoming}
\gamma_-f =g\,,
\end{align}
where $g \geq 0$ is a non-negative measurable function and satisfies
\begin{equation}\label{g-bound}
  \int^T_0\int_{\Sigma_-} g(1+ |v|^2 + |\log g|)\,\bd v\bd \sigma_x\bd t < \infty\quad \text{for any}\quad\! T>0\,.
\end{equation}

In summary, in this paper, we consider the Boltzmann equation \eqref{BE}, with initial condition \eqref{IC-1}-\eqref{IC-bounds}, and boundary condition \eqref{boundary-conditions-incoming}-\eqref{g-bound}.     For the non-cutoff kernel, we work in the class of Alexandre-Villani used \cite{villani2002noncutoff}, i.e. the cross-section satisfies the assumptions from \eqref{AV-first} to \eqref{AV-last}. Our main results are: under these assumptions on the initial-boundary datum and cross-sections, the Boltzmann equation admits a global in time renormalized solution. Furthermore, this solution admits some conservation laws (or inequalities) of mass, momentum, energy and entropy.

\section{Statements of main results}

In this section, we state our main results. The first difficulty we encounter is the definition of renormalized solutions. Besides the renormalization process which is the same as the interior parts for the non-cutoff, in the bounded domain with boundary, the meaning of the ``boundary value" is a nontrivial issue since the solutions lie in the functional space $L^1\cap L\log L$ the element of which can not define the trace in an usual way.  Moreover, for the non-cutoff kernels, the formulation of renormalized solutions needs a defect measure (see \cite{villani2002noncutoff}) which makes the definition of the trace even harder.

The obtained solution in this work just makes sense in the distribution sense, namely in the dual space of smooth test function. There are mainly two kinds of test function space. One is the function space $\mathcal{D}((0,T)\times\mathcal{O})$ which is made up of   smooth function $\phi$ with compact support satisfying
\[ \phi(0, x, v) = \phi(T, x, v) = 0, \qquad \text{for all}~~ (x,v)\in \mathcal{O},  \]
\[ \phi(t, x, v)|_{\partial\Omega} = 0, \qquad \text{for all}~~ (t, v)\in (0,T)\times \mathbb{R}^3,  \]
and there exists $R>0$ such that
\[ \mathrm{Supp}\phi(t,x) \subset \mathrm{B}_{R}(v), \qquad \text{for all}~~ (t,x)\in (0,T)\times\Omega,  \]
where $\mathrm{B}_{R} = \{v | |v| \le R\}$.

The other is function space $\mathcal{D}([0,T]\times\bar{\Omega}\times\mathbb{R}^3)$ which is made up of smooth functions $\phi$ satisfying that
\[ \mathrm{Supp}\phi(t,x) \subset \mathrm{B}_{R}(v), \qquad \text{for all}~~ (t,x)\in [0,T]\times\bar{\Omega}.  \]

In the following,  we will specify the definition of trace for the solution to transport equation while the solution just belongs to $L^1$ space. If the solution to transport equation are smooth up to boundary, then the trace defined below concides with the one in usual sense.
\begin{lemma}[Green Formula{\cite{mischler2010asens}}]
	\label{green-formula-Lp}
	Let $p\in[1, +\infty)$,  $ g \in L^\infty((0,T), L^p_{loc}({\mathcal{O}})$ and $ h \in L^1 ((0,T), L^p_{loc}({\mathcal{O}})$. Assume that $g$ and $h$ satisfies equation
	\begin{align*}
	\partial_t g + v \cdot \nabla_x g = h,
	\end{align*}
	in distribution sense. Then there exists  $\gamma g$ well defined on $(0,T) \times \Sigma$  which satisfies
	\[  \gamma  g \in L^1_{loc} \big([0,T]\times\Sigma, (n(x)\cdot v)^2 dv d\sigma_x dt\big), \]
	and the following Green Formula
	\begin{align*}
	&\int_{0}^{T} \int_{\mathcal{O}} \big( \beta(g)(\partial_t \phi  + v \cdot \nabla_x \phi) + h \beta'(g)\phi\big) \mathrm{d}x \mathrm{d}v \mathrm{d}t \\
	& = \big[ \int_{\mathcal{O}} \beta(g)(\tau, \cdot) \mathrm{d}x \mathrm{d}v \big]\vert_{0}^{T} + \int_{0}^{T} \int\int_{\Sigma} \beta(\gamma g) d\mu d\sigma_x  d\tau,
	\end{align*}
	for  $\beta(.) \in W^{1,\infty}_{loc}(\mathbb{R}^+)$ with $\sup_{x\ge 0}|\beta'(x)| < \infty$,   and all the test function $\phi \in \mathcal{D}_0({[0,T]\times\bar{\Omega}\times\mathbb{R}^3})$, the space of functions $\phi \in \mathcal{D}({[0,T]\times\bar{\Omega}\times\mathbb{R}^3})$ with $\phi\vert_{(0,T)\times\Sigma_0} =0$.
\end{lemma}

Now we introduce the definition of solutions. The renormalized solutions obtained in \cite{villani2002noncutoff} satisfy the following inequality
\[ \partial_t \beta(f) + v \cdot \nabla_x \beta(f) \ge \beta'(f)\Q(f,f) \]
in the sense of distribution for all concave function $\beta$ with  at most logarithm increase rate. Furthermore, we don't know whether $\beta'(f)\Q(f,f)$ belongs to $L^1$ space. So the definition of solution are different with these  on cut-off kernel in \cite{mischler2000-cmp,mischler2010asens}.
\begin{definition}
	\label{def-renormalized}
	Assume that  the cross section $b(z,\omega)$ in \eqref{collision-original} satisfies the assumptions listed from \eqref{AV-first} to \eqref{AV-last} and  $\beta \in C^2(\mathbb{R}^+, \mathbb{R}^+)$ satisfies
\begin{equation}\label{beta-noncutoff}
	 \beta(0) = 0\,,\quad 0 < \beta'(f) < \frac{C}{ 1 +f}\,,\quad\text{and}\quad\! \beta''(f) < 0.
\end{equation}
A nonnegative function
	\[ f \in C\big(\mathbb{R}^+, \mathcal{D}'(\mathcal{O})\big) \cap L^\infty(\mathbb{R}^+; L^1\big((1 + |v|^2 )dxdv )\big)\]
	is called a {\it renormalized solution} to the  Boltzmann equation \eqref{BE}, with initial condition \eqref{IC-1}-\eqref{IC-bounds}, and boundary condition \eqref{boundary-conditions-incoming}-\eqref{g-bound},  if for every renormalization function $\beta$ satisfying \eqref{beta-noncutoff} and every time $T > 0$,
	there is a nonnegative finite defect measure   on $(0, T) \times\mathcal{O}$ such that the following equation
	\begin{align}
	\label{est-definition-renormalized-solution}
	\partial_t \beta(f) + v \cdot \nabla_x \beta(f) \ge \beta'(f)\Q(f,f),
	\end{align}
	holds in the following sense : there exist a trace defined on $(0,T) \times\Sigma_+$ denoted by $\gamma_+ f \in L^1((0,T)\times\Sigma_+)$, and for any non-negative test function $\psi \in \mathcal{D}([0,T]\times \bar{\Omega}\times\mathbb{R}^3)$ with $\psi|_{(0,T)\times\Sigma_0 } = 0$,
	\begin{align}
	\label{est-definition-trace-of-renormalized-defect-measure}
	\begin{split}
	&\int_{0}^{T} \int_{\mathcal{O}} \big( \beta(f)(\partial_t \psi  + v \cdot \nabla_x \psi) +  Q(f, f)\beta'(f) \psi\big) \mathrm{d}x \mathrm{d}v \mathrm{d}t \\
	& \le  \int_{\mathcal{O}} \beta(f)(T, \cdot) \psi  \mathrm{d}x \mathrm{d}v  - \int_{\mathcal{O}} \beta(f)(0, \cdot) \psi  \mathrm{d}x \mathrm{d}v \\
	& + \int_{0}^{T} \int_{\Sigma_+} \beta(\gamma_+ f) \psi  |\mathrm{n}(x) \cdot\cdot v| \bd v \mathrm{d}\sigma_x \mathrm{d}\tau - \int_{0}^{T} \int_{\Sigma_-} \beta(g) \psi  |\mathrm{n}(x) \cdot\cdot v| \bd v  \mathrm{d}\sigma_x \mathrm{d}\tau.
	\end{split}
	\end{align}
	
	Furthermore, $f$, $\gamma_+ f$ and $g$ satisfies the global mass conservation law
	\begin{align*}
	&\int_{\mathcal{O}} f(t) \mathrm{d}v \mathrm{d}x + \int_{0}^{t} \int_{\Sigma_+} \gamma_+ f(s) |\mathrm{n}(x)\cdot v| \mathrm{d}v \mathrm{d}\sigma_x \mathrm{d}s \nonumber \\
	&  = \int_{\mathcal{O}} f_0 \mathrm{d}v \mathrm{d}x + \int_{0}^{t} \int_{\Sigma_-}  g(s) |\mathrm{n}(x)\cdot v| \mathrm{d}v \mathrm{d}\sigma_x \mathrm{d}s, \qquad t \le T.
	\end{align*}
	\end{definition}

\noindent
{\bf Remark:} From \cite{villani2002noncutoff}, since there exists a  defect measure,    the   \eqref{est-definition-renormalized-solution}  which  holds in the sense of distribution is only an inequality.  The only useful information on this defect measure at our disposal is that it is a positive measure.  The   inequality \eqref{est-definition-renormalized-solution} can be multiplied by  positive test function $\psi$ belonging to $\mathcal{D}([0,T]\times\bar{\Omega}\times\mathbb{R}^3)$.  But there are too many candidates $\gamma_+ f$ satisfying \eqref{est-definition-trace-of-renormalized-defect-measure}. It is very natural to assume that at any time $t>0$,  the sum of the mass in the interior domain and the mass on the out going set should be equal to the sum of  initial mass and the mass on the incoming set. Motivated by \cite{villani2002noncutoff}, we introduce the global mass conservation law to define the trace too. Besides, if the solution $f$ are smooth and the defect measures in the interior domain vanishes, by the  the trace of $f$ on $(0,T)\times\Sigma_+$ in the usual sense is equal to $\gamma_+ f$. This is why we denote it by $\gamma_+ f$.

Before stating our main results, we introduce some notations. Let $\mathcal{M}$ be the global Maxwellian, namely, $(2\pi)^{-3} \exp(-\frac{|v|^2}{2}) $. The relative entropy denoted by $ H(f|\mathcal{M})$ is defined as
\[ H(f^n|\mathcal{M}) = \int_\mathcal{O} h(f_{\mathcal{M} }^n) \mathcal{M}\bd v \bd x,   \]
where
\[  h(z) = z\log z - z + 1, z \ge 0, ~f_{\mathcal{M} }^n = \frac{f^n}{\mathcal{M}}.\]
We also denote by $\mathcal{D}(f)$ the {\it H-dissipation}
\[\small 4\mathcal{D}(f) = \int_\Omega \int_{\mathbb{R}^3 \times\mathbb{R}^3}\bd v \bd v_* \bd x \int_{\mathcal{S}^2} \bd \omega B(v-v_*, \omega)(f'f_*' - ff_*)\log \frac{f'f_*'}{ff_*}.\]

\begin{theorem}
	\label{main-result}
  Under the assumption on the cross section $B$   from \eqref{AV-first} to \eqref{AV-last}, if  the initial datum satisfies \eqref{IC-bounds} and  the incoming boundary condition  satisfies \eqref{g-bound},
then  the initial-boundary problem to Boltzmann equation  \eqref{BE} admits a renormalized solution $f$. Furthermore, $f$ has the following properties:
\begin{itemize}
\item Regularity of Trace:
\[ \gamma_+ f \in L^1\big((0,T)\times\Sigma_+; (1+ |v|^2)|\mathrm{n}(x)\cdot v| \mathrm{d}v \mathrm{d}\sigma_x ds\big), ~ \text{for ~ all~} T>0. \]
and
\begin{align*}
\int_{0}^{T} \int_{\Sigma_+}\gamma_+ f |\log \gamma_+ f| \mathrm{d} \mu \mathrm{d}\sigma_x \mathrm{d}s < +\infty,~\text{for ~ all~} T>0.
\end{align*}

\item Local   conservation law of mass:
\begin{align*}
\partial_t \int_{\mathbb{R}^3} f \mathrm{d} v + \nabla \cdot \int_{\mathbb{R}^3} f v \mathrm{d} v = 0,~ \text{in}~~ \mathcal{D}'((0,T)\times\Omega).
\end{align*}

\item Local  conservation law of momentum: There is a distribution-value matrix $W$ belonging to $\mathcal{D}'((0,T)\times\Omega)$   such that
\begin{align*}
\partial_t \int_{\mathbb{R}^3} v f(t)\mathrm{d}v + \nabla \cdot   \int_{\mathbb{R}^3} v\otimes v f  \mathrm{d}v  + \nabla \cdot W  =0 ,~~\text{in},~~~ \mathcal{D}'((0,T)\times\Omega).
\end{align*}
\item Global momentum conservation law:
\begin{align}
\begin{split}
&\int_{\mathcal{O}} f(t) v \mathrm{d}v \mathrm{d}x + \int_{0}^{t} \int_{\Sigma_+} v \gamma_+ f(s) |\mathrm{n}(x)\cdot v| \mathrm{d}v \mathrm{d}\sigma_x \mathrm{d}s  \\
&  = \int_{\mathcal{O}} f_0  v \mathrm{d}v \mathrm{d}x + \int_{0}^{t} \int_{\Sigma_-} v  g(s) |\mathrm{n}(x)\cdot v| \mathrm{d}v \mathrm{d}\sigma_x \mathrm{d}s, \qquad t \le T.
\end{split}
\end{align}

\item Global energy  inequality:
\begin{align}
\begin{split}
&\int_{\mathcal{O}} f(t) |v|^2  \mathrm{d}v \mathrm{d}x + \int_{0}^{t} \int_{\Sigma_+} |v|^2 \gamma_+ f(s) |\mathrm{n}(x)\cdot v| \mathrm{d}v \mathrm{d}\sigma_x \mathrm{d}s  \\
&  \le \int_{\mathcal{O}} f_0 |v|^2 \mathrm{d}v \mathrm{d}x + \int_{0}^{t} \int_{\Sigma_-}  g(s) |v|^2 |\mathrm{n}(x)\cdot v| \mathrm{d}v \mathrm{d}\sigma_x \mathrm{d}s, \qquad t \le T.
\end{split}
\end{align}	

\item Global entropy inequality:
\begin{align}
\label{est-theorem-main-entropy}
\begin{split}
& H(f|\mathcal{M})(t) + \int_{0}^{t}\int_{\Sigma_+}h(\gamma_+ f_{\mathcal{M} } ) |\mathrm{n}(x)\cdot v| \mathrm{d}v \mathrm{d}\sigma_x\mathrm{d}s +  \int_{0}^t \mathcal{D}(f)(s) \mathrm{d}x \mathrm{d}v \mathrm{d}s \\
&\le  H(f_0|\mathcal{M})  + \int_{0}^{t}\int_{\Sigma_-}h(g|{\mathcal{M} } ) |\mathrm{n}(x)\cdot v| \mathrm{d}v \mathrm{d}\sigma_x\mathrm{d}s,\qquad t \le T.
\end{split}
\end{align}
\end{itemize}
\end{theorem}
\begin{remark}
As for the local mass conservation law, similar to Lemma \ref{green-formula-Lp},  we can use the Green formula to define the trace of $\int_{\mathbb{R}^3} v f\mathrm{d} v$ on $\partial\Omega$. {  Denote it by $\gamma_x(\int_{\mathbb{R}^3} v f\mathrm{d} v)$,
$$ \mathrm{n}(x) \gamma_x(\int_{\mathbb{R}^3} v f\mathrm{d} v) = \int_{\Sigma_+^x} \gamma_+ f  | \mathbf{n}(x) \cdot v  |\mathrm{d} v  - \int_{\Sigma_-^x} \gamma_- f |\mathrm{n}(x) \cdot v| \mathrm{d} v,$$ and $$\mathrm{n}(x)\gamma_x(\int_{\mathbb{R}^3} f v \mathrm{d} v)  \in L^1\big((0,T)\times\partial\Omega; \mathrm{d} \sigma_x \mathrm{d} s \big).$$}
\end{remark}

\begin{remark}
 This result also works for   unbounded domain case. While   on the unbounded domain, the weight $|x|^2$ are necessary.  Besides, all these result are still correct in $\mathbb{R}^n$, $n\ge 3$.
\end{remark}

Compared  to  Boltzmann equation with cutoff kernel, from \cite{villani2002noncutoff}, the solution $f$ only satisfies the following inequality
\[ \partial_t \beta(f) + v \cdot \nabla_x \beta(f) \ge \beta'(f)\Q(f,f) \]
in distributional sense. Besides, $\beta(f)\Q(f,f)$ doesn't belongs to $L^1$ space. It is just a distribution belonging to $\mathcal{D}'((0,T)\times\mathcal{O})$. So the trace theory in \cite{mischler2010asens,mischler2000vlasovtrace} and references therein on cutoff case completely doesn't work here. It needs some new idea. Since the trace on $\Sigma_-$ is fixed, the main task is to find some $\gamma_+f$ satisfying \eqref{est-definition-trace-of-renormalized-defect-measure}, conservation law of mass and Theorem \ref{main-result}. From our former work on Boltzmann equation with cutoff kernel, we can construct a sequence of approximate solutions whose traces are weakly compact in $L^1$ space. Noticing that the renormalized function $\beta$ are convex and the test function are positive, then by the upper semi-continuity of convex function, we complete the proof.

\section{Estimates of Approximate system}
\label{sec-approximate}
In this section, we will construct a sequence of approximate solutions to Boltzmann equation with modified collision kernel $\Q^n$, namely
\begin{equation}
\label{collision-original-cut}
\begin{split}
\Q^n(f,f)=\big(\frac{1}{\footnotesize 1 + \frac{1}{n}\int f^n dv}\big)\int_{\mathbb{R}^3\times\mathbb{S}^2}B_n |[f(v')f({v_*}')-f(v)f(v_*)]\bd v_*\bd \omega
\end{split}
\end{equation}
with \begin{align}
\label{cross-section}
B_n(v-v_*,\omega) = B(v-v_*, \omega) \cdot I_{ \frac{1}{n} \le  \theta \le \pi}  I_{ \frac{1}{n} \le  |v - v_*|  \le n^2}.
\end{align}
For every $n \in \mathbb{N}^+$, the initial data of approximate system  are chosen as the one in \cite{lions1989cauchy}, namely \begin{align}
\label{app-f-0}
f_0^n = \tilde{f}_0^n + \frac{1}{n}\exp(-\frac{|x|^2}{2} - \frac{|v|^2}{2}),
\end{align}
where $\tilde{f}_0^n$ is obtained by truncating $f_0$ first and then smoothing it.   In details,   we will solve the following initial-boundary problem
\begin{align}
\label{boltzmann-approximate-lions}
\begin{cases}
\partial_t f^n + v \cdot \nabla_x f^n = \Q^n(f^n, f^n),~ x \in \Omega \subset \mathbb{R}^3,  v\in \mathbb{R}^3, \\
f^n(0,x, v) = f_0^n(x, v),\\
\gamma_-f^n =g,~~\text{on}~~\Sigma_-
\end{cases}
\end{align}
where $g$ satisfies
\begin{align}
\label{app-g}
\int_{0}^{t} \int_{\Sigma_-} g(1 + |v|^2 + |\log g|)  \mathrm{d}\mu \mathrm{d}\sigma_x \mathrm{d}s < C(t) < \infty,~~\text{for all} ~ t >0.
\end{align}
and $f_0^n$ satisfies
\begin{align}
\label{app-f}
\int_{\mathcal{O}} f_0^n(1 + |v|^2 + |\log f_0^n|)  \mathrm{d}\mu \mathrm{d}\sigma_x \mathrm{d}s < C_0 < \infty,~~\text{for all}~ n.
\end{align}

For each fixed $n$, we can use fixed point theorem to solve system \eqref{boltzmann-approximate-lions}.  Th detailed proof of the existence can been found in our former work where we obtain the following theorem about global existence:

\begin{theorem}[global-in-time existence]
	\label{theorem-boltzmann-approximate-lions-global} For any $T>0$, under the assumption \eqref{app-f} and \eqref{app-g},  for every $n$,  system \eqref{boltzmann-approximate-lions} has a unique solution $f^n \in L^\infty([0,T];L^1(\mathcal{O}))$ such that
	\[ \partial_t f^n + v \cdot \nabla_x f^n = \Q^n(f^n, f^n) \]
	hold in the sense of distribution. Further, there exists  a unique trace $\gamma_+ f \in L^1((0,T)\times\Sigma_+;\bd \mu \bd \sigma_x \bd s) $ to \eqref{boltzmann-approximate-lions}  such that
	\begin{align*}
	&\int_{0}^{T} \int_{\mathcal{O}} \big( \beta(f^n)(\partial_t \phi  + v \cdot \nabla_x \phi) + \Q^n(f^n,f^n) \beta'(f^n)\phi\big) \mathrm{d}x \mathrm{d}v \mathrm{d}t \\
	& =   \int_{\mathcal{O}}  \phi(T) \beta(f^n)(T) \mathrm{d}x \mathrm{d}v   -    \int_{\mathcal{O}}  \phi(0) \beta(f^n_0)  \mathrm{d}x \mathrm{d}v   \\
	&  + \int_{0}^{T}  \int_{\Sigma_+} \phi \beta(\gamma_+ f^n) \bd \mu \bd \sigma_x  \bd t - \int_{0}^{T}  \int_{\Sigma_-} \beta( g) \phi  \bd \mu \bd \sigma_x  \bd t,
	\end{align*}
	for all  $\beta'(.) \in L^\infty(\mathbb{R}^+)$   and all the test function $\phi \in \mathcal{D}_0({[0,T]\times\bar{\Omega}\times\mathbb{R}^3})$, the space of functions $\phi \in \mathcal{D}({[0,T]\times\bar{\Omega}\times\mathbb{R}^3})$ with $\phi\vert_{(0,T)\times\Sigma_0} =0$.
	Furthermore, $f^n$ and $\gamma_+ f^n$ satisfy
	\begin{itemize}
		\item global conservation law of mass:
			\begin{align}
		\label{est-theorem-approximate-lions-global-density}
		\begin{split}
	&\int_{\mathcal{O}}   f^n(t) \mathrm{d}x \mathrm{d}v + \int_{0}^{t}\int_{\Sigma_+}  \gamma_+ f^n |\mathrm{n}(x)\cdot v| \mathrm{d}v \mathrm{d}\sigma_x\mathrm{d}s \\
	&=  \int_{\mathcal{O}} f^n_0 \mathrm{d}x \mathrm{d}v  +  \int_{0}^{t}\int_{\Sigma_-} g |\mathrm{n}(x)\cdot v| \mathrm{d}v \mathrm{d}\sigma_x\mathrm{d}s,\qquad t \le T.
	\end{split}
		\end{align}
			\item global conservation law of momentum
			\begin{align}
			\label{est-theorem-approximate-lions-global-mo}
			\begin{split}
			&\int_{\mathcal{O}}  v f^n(t) \mathrm{d}x \mathrm{d}v + \int_{0}^{t}\int_{\Sigma_+} v \gamma_+ f^n |\mathrm{n}(x)\cdot v| \mathrm{d}v \mathrm{d}\sigma_x\mathrm{d}s \\
			& =  \int_{\mathcal{O}} v f^n_0 \mathrm{d}x \mathrm{d}v  +  \int_{0}^{t}\int_{\Sigma_-} v g |\mathrm{n}(x)\cdot v| \mathrm{d}v \mathrm{d}\sigma_x\mathrm{d}s, \qquad t \le T.
			\end{split}
			\end{align}
			\item global conservation las of energy
			\begin{align}
			\label{est-theorem-approximate-lions-global-energy}
			\begin{split}
			&\int_{\mathcal{O}}  | v|^2f^n(t) \mathrm{d}x \mathrm{d}v + \int_{0}^{t}\int_{\Sigma_+} | v|^2\gamma_+ f^n |\mathrm{n}(x)\cdot v| \mathrm{d}v \mathrm{d}\sigma_x\mathrm{d}s \\
			&\le  \int_{\mathcal{O}}  | v|^2f^n_0 \mathrm{d}x \mathrm{d}v  +  \int_{0}^{t}\int_{\Sigma_-}| v|^2g |\mathrm{n}(x)\cdot v| \mathrm{d}v \mathrm{d}\sigma_x\mathrm{d}s, \qquad t \le T.
			\end{split}
			\end{align}
			\item global entropy inequality
			\begin{align}
			&\label{est-theorem-app-lions-entropy}
			 \int_{\mathcal{O}}  f^n\log f^n(t) \mathrm{d}x \mathrm{d}v + \int_{0}^{t}\int_{\Sigma_+}\gamma_+f^n  \log \gamma_+ f^n |\mathrm{n}(x)\cdot v| \mathrm{d}v \mathrm{d}\sigma_x\mathrm{d}s +  \int_{0}^t \mathcal{D}(f^n)(s) \mathrm{d}x \mathrm{d}v \mathrm{d}s  \nonumber \\
			& \le \int_{\mathcal{O}}  f^n_0\log f^n_0 \mathrm{d}x \mathrm{d}v + \int_{0}^{t}\int_{\Sigma_-}g \log g |\mathrm{n}(x)\cdot v| \mathrm{d}v \mathrm{d}\sigma_x\mathrm{d}s,\qquad  t \le T,
			\end{align}
			\item global relative entropy inequality
			\begin{align}
			&    H(f^n|\mathcal{M})(t) + \int_{0}^{t}\int_{\Sigma_+}h(\gamma_+ f_{\mathcal{M} }^n ) |\mathrm{n}(x)\cdot v| \mathrm{d}v \mathrm{d}\sigma_x\mathrm{d}s +  \int_{0}^t \mathcal{D}(f^n)(s) \mathrm{d}x \mathrm{d}v \mathrm{d}s\nonumber\\
			& \label{est-theorem-app-lions-global-relative-entropy} \le  H(f_0^n|\mathcal{M})  + \int_{0}^{t}\int_{\Sigma_-}h(g|{\mathcal{M} } ) |\mathrm{n}(x)\cdot v| \mathrm{d}v \mathrm{d}\sigma_x\mathrm{d}s,\qquad   t  \le T.
			\end{align}
	\end{itemize}

\end{theorem}
\begin{remark}
	From \eqref{est-theorem-approximate-lions-global-energy} and
	we can infer that
	\begin{align}
		\label{est-theorem-app-lions-entropy-abs}
		\begin{split}
	\sup\limits_{0\le s \le t}\int_{\mathcal{O}}  f^n|\log f^n|(s) \mathrm{d}x \mathrm{d}v + \int_{0}^{t}\int_{\Sigma_+}\gamma_+f^n |\gamma_+ \log f^n| |\mathrm{n}(x)\cdot v| \mathrm{d}v \mathrm{d}\sigma_x\mathrm{d}s \le C(T).
	 \end{split}
	\end{align}
\end{remark}

\section{Estimate of Renormalized Formulation for the non-cutoff case}
\label{sec-est-renormalized-formulation}
This section is devoted to using the conservation law to bound the source terms $\beta'(f^n) Q^n(f^n, f^n)$ where the $\beta(\cdot)$ satisfies the Definition \ref{def-renormalized} while the cross section satisfies all the assumption from \eqref{AV-first} to \eqref{AV-last}. To simplify the notations, we drop the superscript for the time being and just prove $\beta'(f) Q(f, f)$ can be controlled by
\begin{align}
\label{conservations}
\sup_{0\le s \le T} \int f(s)(1   + |v|^2 + |\log f|) \mathrm{d}x \mathrm{d}v  \le C(T),~~~~\forall t>0.
\end{align}
Then we will show how to modify three lemmas to the approximate case $\beta'(f^n) Q^n(f^n, f^n)$.
As \cite{villani2002noncutoff}, we split $\beta'(f)\Q(f,f)$ into three parts
\begin{align}
\label{decompse}
\beta'(f)\Q(f,f) = \mathcal{R}_1 + \mathcal{R}_2 + \mathcal{R}_3,
\end{align}
where
\begin{align}
&\label{decompose-r-1} \mathcal{R}_1 = [ f\beta'(f) - \beta(f) ]\int_{\mathcal{R}^3\times S^2}dv_* d\omega B(f_*' - f_*),\\
&\label{decompose-r-2} \mathcal{R}_2 =\int_{\mathcal{R}^3\times S^2}dv_* d\omega B[f_*' \beta(f')- f_*\beta(f)],\\
&\label{decompose-r-3} \mathcal{R}_3 =\int_{\mathcal{R}^3\times S^2}dv_* d\omega Bf_*'\big(\beta(f') - \beta(f) - \beta'(f)(f'-f)\big).
\end{align}
Moreover,
\begin{align}
\label{decompose-r-1-rewrite}
\mathcal{R}_1 = [ f\beta'(f) - \beta(f) ] f*_v \mathcal{S},
\end{align}
where
\[ \mathcal{S}(|z|) = C \int_{0}^{\pi/2} d\theta \sin\theta \big[  \frac{1}{\cos\frac{\theta}{2}} B(\frac{1}{\cos\frac{\theta}{2}},\cos \theta ) - B(|z|, \cos \theta)  \big],   \]
and
\begin{align}
\label{decompose-r-2-rewrite}
\int_{\mathbb{R}^3} \mathcal{R}_2 \phi(v) dv = \int_{\mathbb{R}^3\times\mathbb{R}^3} dv dv_* f(v_*) \beta(f)\mathcal{T}(\phi),
\end{align}
where
\[  \mathcal{T}(\phi) = \int_{S^1} B(v-v_*, \omega)(\phi'-\phi) d\omega = 2\pi\int_{0}^\pi B(|v-v_*|, \cos \theta)(\phi'-\phi) \sin \theta d \theta . \]
From \cite{villani2002noncutoff}, we know that $\mathcal{S}(|v|)$ and $\mathcal{T}(\phi)$ have very good properties even thought there exist angular singularity and velocity singularity for $B$. This owes to the {\it Symmetry-Induced Cancellation effects}.
\begin{proposition}
	\label{prop-s}
	Let $B$ satisfies assumption from \eqref{AV-first} to \eqref{AV-last}, then $\mathcal{S}$ is local integrable and $\mathcal{S}(|z|) = o(|z|^2),~~ |z| \to \infty$ .
\end{proposition}

\begin{proposition}
	\label{prop-t}
		Let $B$ satisfies assumptions from \eqref{AV-first} to \eqref{AV-last}, then for all $\phi \in W^{2,\infty}(\mathbb{R}^3)$, we have
		\[  |\mathcal{T}(\phi)| \le C \|\phi\|_{W^{2,\infty}} |v-v_*|( 1 + |v -v_*| )M(|v-v_*|). \]
\end{proposition}

Alexandre and Villani dealt with the whole space case. In this work, we consider the initial-boundary case. For $\mathcal{R}_1$ and $\mathcal{R}_2$,  ways of  using   conservation quantities \eqref{conservations}, Proposition \ref{prop-s} and Proposition \ref{prop-t}  to control  $\mathcal{R}_1$ and $\mathcal{R}_2$ in the whole space works for the bounded case.  Similarly, we can infer
\begin{lemma}
	\label{lemma-decompose-r-1}
	Let $B$ satisfy assumptions from \eqref{AV-first} to \eqref{AV-last}, and let $f$ satisfy \eqref{conservations}. Then $\mathcal{R}_1$ lies in the function space $L^\infty([0,T]; L^1(\Omega \times B_R(v)))$ for all $R>0$, where $B_R(v) = \{ v \in \mathbb{R}^3 \vert ~~ |v| \le R \}$.
\end{lemma}
\begin{lemma}
	\label{lemma-decompose-r-2}
	Let $B$ satisfy assumptions from \eqref{AV-first} to \eqref{AV-last}. Then $\mathcal{R}_2$ belongs to  the function space $L^\infty([0,T]; L^1(\Omega; W^{-2,1}(B_R(v))))$ for all $R>0$.
\end{lemma}
The term $\mathcal{R}_3$ is different. In \cite{villani2002noncutoff}, they work on the whole space. The advantage is that they don't need to consider the boundary effect. In details, after multiplying  \eqref{decompse} by some proper test function $\phi(v)$ and integrating over $\mathbb{R}^3 \times \mathbb{R}^3$,
\begin{align}
\label{different-with-whole-space-case}
\int_{\mathbb{R}^3\times\mathbb{R}^3} v \cdot \nabla_x \beta(f) \cdot \phi(v) \mathrm{d}v \mathrm{d}x = 0.
\end{align}
But when we consider the initial-boundary problem,
\begin{align}
\label{different-with-boundary-case}
\int_{\Omega\times\mathbb{R}^3} v \cdot \nabla_x \beta(f) \cdot \psi(v) \mathrm{d}x \mathrm{d}v \neq 0.
\end{align}
But on the other hand, from  \eqref{est-theorem-approximate-lions-global-density} and \eqref{est-theorem-approximate-lions-global-energy},  recalling $\gamma_- f = g$, we can get
\begin{align}
\label{conservation-boundary}
\int_{0}^{t}\int_{\Sigma_+} (1 + |v|^2)\gamma_+ f |\mathrm{n}(x)\cdot v| \mathrm{d}v \mathrm{d}\sigma_x \mathrm{d}s + \int_{0}^{t}\int_{\Sigma_-} (1 + |v|^2)\gamma_-f |\mathrm{n}(x)\cdot v| \mathrm{d}v \mathrm{d}\sigma_x \mathrm{d}s \le C(T).
\end{align}
Then $\mathcal{R}_3$ can be controlled as following:
\begin{lemma}
	\label{lemma-decompose-r-3}
	Let $B$ satisfy assumption  from \eqref{AV-first} to \eqref{AV-last}, and let $f$ satisfy \eqref{conservations} and \eqref{conservation-boundary}. Then $\mathcal{R}_3$ lies in $L^\infty([0,T]; L^1(\Omega \times B_R(v)))$ for all $R>0$.
\end{lemma}
\begin{proof}
	Firstly, from \eqref{decompse},
	 \begin{align*}
	 \partial_t \beta(f) + v \cdot \nabla_x \beta(f) \ge  \mathcal{R}_1 + \mathcal{R}_2 + \mathcal{R}_3,
	 \end{align*}
	 multiplying the above equation with $\phi(v)$, satisfying $ \phi(v)|_{B_R} = 1, \phi(v) \ge 0,~ \phi(v) = 0, |v| \ge 2R$, then
	 integrating the resulting equation over $\mathcal{O}$,
	 \begin{align*}
	 \int_{0}^{t} \int_{\mathcal{O}}\mathcal{R}_3(s) \phi(v)  \mathrm{d}v \mathrm{d}x ds & \le - \int_{0}^{t} \int_{\mathcal{O}} \phi \mathcal{R}_2 \mathrm{d}v \mathrm{d}x ds - \int_{0}^{t} \int_\mathcal{O} \phi \mathcal{R}_2 \mathrm{d}v \mathrm{d}x ds \\
	 & + \int_{\mathcal{O}} \beta(f)(t) dv ds - \int_{\mathcal{O}} \beta(f_0) dvds \\
	 & + \int_{0}^t \int_{\mathcal{O}}v \cdot \nabla_x (\beta(f) \phi) \mathrm{d}v \mathrm{d}x ds\\
	 & :=  A_1(t) + A_2(t) + A_3(t) + A_4(t) + A_5(t).
	 \end{align*}
Recalling that \[ 0 < \beta(f) \le f,\]
using Lemma \ref{lemma-decompose-r-1} and Lemma \ref{lemma-decompose-r-2}, then there exists some constant dependent on $R$ and $T$ such that
\begin{align}
\label{est-lemma-r-3-a-1-to-4}
\sum_{i=1}^4 |A_i(t)| \le C(T, R),~~ t \le T.
\end{align}	
$A_5$ is more complicated. Integrating it by parts, recalling
\[ 0 \le \phi(v) \le 1, ~~ 0 \le \beta(\gamma f ) \le \gamma f, \gamma_\pm f \ge 0,  \]
 we have
\begin{align}
\begin{split}
\label{est-lemma-r-3-a-5}
\int_{0}^t v \cdot \nabla_x (\beta(f) \phi) \mathrm{d}v \mathrm{d}x ds = &  - \int_{0}^{t} \int_{\Sigma_-} \beta(\gamma_-f)(t, x,v) \phi(v) |\mathrm{n}(x)\cdot v| \mathrm{d}v \mathrm{d}\sigma_x \mathrm{d}s \\
& + \int_{0}^{t} \int_{\Sigma_+} \beta(\gamma_+ f)(t, x,v) \phi(v) |\mathrm{n}(x)\cdot v| \mathrm{d}v \mathrm{d}\sigma_x \mathrm{d}s\\
& \le \int_{0}^{t} \int_{\Sigma_-} \gamma_- f(t, x,v) |\mathrm{n}(x)\cdot v| \mathrm{d}v \mathrm{d}\sigma_x \mathrm{d}s \\
& + \int_{0}^{t} \int_{\Sigma_+} \gamma_+f(t, x,v)  |\mathrm{n}(x)\cdot v| \mathrm{d}v \mathrm{d}\sigma_x \mathrm{d}s\\
& \le C(T),
\end{split}
\end{align}	
where we have used \eqref{conservation-boundary}.

Summing \eqref{est-lemma-r-3-a-1-to-4} and \eqref{est-lemma-r-3-a-5} up, recalling that
\[ \mathcal{R}_3 \ge 0,   \]
we complete the proof of this lemma.

	\begin{remark}
	\label{remark-renormalized-formulation-approximate}
	Noticing that the truncated cross section $B_n$ also satisfies assumption \eqref{Grad-cutoff-1} and \eqref{Grad-cutoff-2} and $0 < \big(\frac{1}{\footnotesize 1 + \frac{1}{n}\int f^n dv}\big) \le 1$ , so in the similar way,  we can prove that   the corresponding  $\mathcal{R}_1^n, \mathcal{R}_2^n$ and $\mathcal{R}_3^n$ corresponding to \eqref{boltzmann-approximate-lions} ,
	\begin{align*}
	& \mathcal{R}_1^n = [ f^n\beta'(f^n) - \beta(f^n) ]\int_{\mathbb{R}^3\times S^2}dv_* d\omega B_n(f_*^{n'} - f_*),\\
	& \mathcal{R}_2^n =\int_{\mathbb{R}^3\times S^2}dv_* d\omega B_n[f_*^{n'} \beta(f^{n'})- f_*^n\beta(f^n)],\\
	& \mathcal{R}_3^n=\int_{\mathbb{R}^3\times S^2}dv_* d\omega B_n f_*^{n'}\big(\beta(f^{n'}) - \beta(f^n) - \beta'(f^n)(f^{n'}-f^n)\big).
	\end{align*}
	
	also satisfy   Lemma \ref{lemma-decompose-r-1}, Lemma \ref{lemma-decompose-r-2} and Lemma \ref{lemma-decompose-r-3}.
\end{remark}

\end{proof}

\section{Weak compactness and global existence }
This whole section is devoted to  prove Theorem \ref{main-result}. The key tools is $L^1$ weak compactness theorem, Dunford-Pettis Lemma and De la Vall\'ee-Poussin uniform integrability criterion. One can check \cite{evans2015measureandfinepropertiesoffunctions} for details.  We need to consider the interior parts(Theorem \ref{theorem-interior}) and boundary parts. The interior parts have been done by Alexandre and Villani in \cite{villani2002noncutoff}. We just quote it.

\subsection{Interior domains}
From estimates \eqref{est-theorem-approximate-lions-global-density}, \eqref{est-theorem-approximate-lions-global-energy} and \eqref{est-theorem-app-lions-entropy-abs}, using Dunford-Pettis Lemma, we can conclude that there exist some $f$ such that
\begin{align}
\label{weak-convergence-f}
f^n \rightharpoonup f,~~L^1 ((0,T) \times \mathcal{O}).
\end{align}

Moreover, by simple calculation, the approximate cross section chosen $B_n$ in \eqref{cross-section} also   satisfies the following assumption.
For all $n$,
\begin{align}
\label{cross-section-singularity-1}
B_n(z, \omega) \ge \Phi_0(|z|)b_{0,n}(\kappa \cdot \omega),~~ \kappa = \frac{z}{|z|},
\end{align}
for some fixed continuous function $\Phi_0(|z|)$ such that $\Phi_0(|z|)>0$ if $z\neq 0$, and
\begin{align}
\label{cross-section-singularity-2}
\int_{S^2} \liminf_{n\to \infty} b_{0,n}(\kappa\cdot \omega) d\omega = \infty.
\end{align}

\begin{theorem}[Extended Stability \cite{villani2002noncutoff}]
	\label{theorem-interior}
Let $B$ satisfy assumption from \eqref{AV-first} to \eqref{AV-last}. Let $(f^n)_{n\in \mathbb{N}}$ be a sequence of solutions to \eqref{boltzmann-approximate-lions} with initial data \eqref{app-f} and boundary conditions in Theorem \ref{theorem-boltzmann-approximate-lions-global} and   satisfy the natural {\it \'a priori} bounds \eqref{conservations} and \eqref{conservation-boundary}. Assume without loss generality that
\[ f^n \rightharpoonup f ~\text{in}~ L^p([0,T], L^1(\Omega\times\mathbb{R}^3)),~~1 \le p < \infty.  \]
Then
\begin{enumerate}
	\item $f^n \to f$ in  $L^p([0,T], L^1(\Omega\times\mathbb{R}^3))$;
	\item for all functions $\beta \in C^2(\mathbb{R}^+, \mathbb{R}^+)$ satisfying
	\[ \beta(0) = 0, 0 < \beta'(f) < \frac{C}{ 1 +f}, \beta''(f) < 0,  \]	
	there exists a defect measure $\nu$ such that the  equality
	\[   \partial_t \beta(f) + v \cdot \nabla_x \beta(f) = \beta'(f)\Q(f,f) + \nu  \]
	holds in the sense of distributions.
	\item for each $\phi \in \mathcal{D}([0,T]\times\bar{\Omega} \times\mathbb{R}^3)$,
	\[ \int_0^T \int_\Omega \mathcal{R}_1^n \phi \bd x \bd v \bd s \to \int_0^T \int_\Omega \mathcal{R}_1 \phi \bd x \bd v \bd s,  \]
	\[ \int_0^T \int_\Omega \mathcal{R}_2^n \phi \bd x \bd v \bd s \to \int_0^T \int_\Omega \mathcal{R}_2 \phi \bd x \bd v \bd s,  \]
	and moreover, if $\phi$ is a non-negative function,
	\[ \int_0^T \int_\Omega \mathcal{R}_3 \phi \bd x \bd v \bd s \le \liminf_{n\to \infty} \int_0^T \int_\Omega \mathcal{R}_3^n \phi \bd x \bd v \bd s.  \]
\end{enumerate}
\end{theorem}
\begin{proof}
For the first two items,  they directly come from \cite{villani2002noncutoff}.   Indeed,    multiplying the first equation of \eqref{boltzmann-approximate-lions} by $\phi \in D((0,T)\times\mathcal{O})$,
 	\begin{align*}
 	&\int_{0}^{T} \int_{\mathcal{O}} \big( \beta(f^n)(\partial_t \phi  + v \cdot \nabla_x \phi) + Q^n(f^n, f^n)\beta'(f^n) \phi\big) \mathrm{d}x \mathrm{d}v \mathrm{d}t = 0,
 	\end{align*}
 then   the left proof is similar to \cite{villani2002noncutoff}.	Here we focus on the defect measure $\nu$.

Recalling
 \[    Q^n(f^n, f^n)\beta'(f^n) = \mathcal{R}_1^n + \mathcal{R}_2^n + \mathcal{R}_3^n, \]
 with the help of Remark \ref{remark-renormalized-formulation-approximate},  we can find that $(\mathcal{R}_3^n \bd x \bd v \bd s)$ is a bounded measure-value sequence on $\mathcal{D}'((0,T)\times\mathcal{O})$. Then there exist some measure $\bd \mathbf{m}$ such that
 \[ \int_{0}^{T} \int_{\mathcal{O}}   \mathcal{R}_3^n \phi  \mathrm{d}x \mathrm{d}v \mathrm{d}s  \to \int_{0}^{T} \int_{\mathcal{O}}    \phi \bd \mathbf{m},~~~~\text{for any} ~\phi ~\in~ \mathcal{D}((0,T)\times\mathcal{O}).   \]
 Then
 \[ \bd \nu = \bd \mathbf{m} - \mathcal{R}_3   \mathrm{d}x \mathrm{d}v \mathrm{d}s.  \]
If $\phi \ge 0$,  by Fatou Lemma,   we can deduce that $\bd \nu$ is a positive measure.

As for the third entry, noticing that Lemma \ref{lemma-decompose-r-1} and Lemma \ref{lemma-decompose-r-2} only require that the velocity variable of test function has compact support. Using the strong compactness of $f^n$, the third item can be verified too by the argument in \cite{villani2002noncutoff}.

\end{proof}
\subsection{Boundary parts}
The boundary parts are more complicated. First, Lemma \ref{green-formula-Lp} works only for the equality. But the solution in Theorem \ref{theorem-interior} is just a inequality.  So it is rarely possible to define trace of solutions in Definition \ref{def-renormalized}  by Lemma \ref{green-formula-Lp}. In the meantime, we have defined a meaningful trace $\gamma_\pm f^n$ for the approximate solution $f^n$. The good new is that the trace sequence is also weakly compact in $L^1$ space, namely, according to \eqref{est-theorem-app-lions-entropy-abs}, \eqref{est-theorem-approximate-lions-global-density} and \eqref{est-theorem-approximate-lions-global-energy}, by Dunford-Pettis Lemma,   we can infer: there exists some $f_\gamma \in L^1((0,T) \times \Sigma_+, |\mathrm{n}(x)\cdot v| \mathrm{d}v \mathrm{d}\sigma_x \mathrm{d}s )$ such that
\begin{align*}
& \gamma_+ f^n \rightharpoonup f_\gamma,~~ L^1((0,T) \times \Sigma_+, |\mathrm{n}(x)\cdot v| \mathrm{d}v \mathrm{d}\sigma_x \mathrm{d}s ),\\
& \int_{0}^{t}\int_{\Sigma_+}   \gamma_+ f^n |\mathrm{n}(x)\cdot v| \mathrm{d}v \mathrm{d}\sigma_x\mathrm{d}s \to  \int_{0}^{t}\int_{\Sigma_+}    f_\gamma |\mathrm{n}(x)\cdot v| \mathrm{d}v \mathrm{d}\sigma_x\mathrm{d}s.
\end{align*}
On the other hand, for the approximate solutions $(f^n)$ and its trace $\gamma_+ f^n$,  recalling that
\begin{align}
	\label{est-app-lions-density-cp}
	\begin{split}
	& \int_{\mathcal{O}} f^n(t) \mathrm{d}x \mathrm{d}v + \int_{0}^{t}\int_{\Sigma_+}   \gamma_+ f^n |\mathrm{n}(x)\cdot v| \mathrm{d}v \mathrm{d}\sigma_x\mathrm{d}s \\
	& =    \int_{\mathcal{O}}  f_0 \mathrm{d}x \mathrm{d}v + \int_{0}^{t}\int_{\Sigma_-}   g  , ~ t > 0.
	\end{split}
	\end{align}
By theorem \ref{theorem-interior},  \[ f^n \rightharpoonup f, ~~ L^1(\mathcal{O}),\]
thus we get
\begin{align*}
\int_{\mathcal{O}} f^n(s) \mathrm{d}x \mathrm{d}v \to \int_{\mathcal{O}} f(s) \mathrm{d}x \mathrm{d}v,\qquad s \le T.
\end{align*}
Then
similarly, we can deduce
\begin{align*}
& v \gamma_+  f^n \rightharpoonup v f_\gamma,~~ L^1((0,T) \times  \Sigma_+, |\mathrm{n}(x)\cdot v| \mathrm{d}v \mathrm{d}\sigma_x \mathrm{d}s ),\\
& \int_{0}^{t}\int_{\Sigma_+}  v \gamma_+ f^n |\mathrm{n}(x)\cdot v| \mathrm{d}v \mathrm{d}\sigma_x\mathrm{d}s \to  \int_{0}^{t}\int_{\Sigma_+}   v f_\gamma |\mathrm{n}(x)\cdot v| \mathrm{d}v \mathrm{d}\sigma_x\mathrm{d}s.
\end{align*}
Recalling that $(1 + |v|)f_0^n \to (1 + |v|) f_0,~\text{in}~~ L^1(\mathcal{O})$,~thus, we infer that
	\begin{align}
	\label{est-lions-density-0}
	\begin{split}
	& \int_{\mathcal{O}} f \mathrm{d}x \mathrm{d}v + \int_{0}^{t}\int_{\Sigma_+}   f_\gamma |\mathrm{n}(x)\cdot v| \mathrm{d}v \mathrm{d}\sigma_x\mathrm{d}s \\
	& =    \int_{\mathcal{O}}  f_0 \mathrm{d}x \mathrm{d}v + \int_{0}^{t}\int_{\Sigma_-}   g |\mathrm{n}(x)\cdot v| \mathrm{d}v \mathrm{d}\sigma_x\mathrm{d}s, \qquad t > 0,
	\end{split}
	\end{align}
	and
	\begin{align}
\label{est-lions-m}
\begin{split}
& \int_{\mathcal{O}} v f \mathrm{d}x \mathrm{d}v + \int_{0}^{t}\int_{\Sigma_+}  v f_\gamma |\mathrm{n}(x)\cdot v| \mathrm{d}v \mathrm{d}\sigma_x\mathrm{d}s \\
& =    \int_{\mathcal{O}} v f_0 \mathrm{d}x \mathrm{d}v + \int_{0}^{t}\int_{\Sigma_-} v  g |\mathrm{n}(x)\cdot v| \mathrm{d}v \mathrm{d}\sigma_x\mathrm{d}s, \qquad t > 0.
\end{split}
\end{align}

The left goal is to show that $f_\gamma$ satisfies \eqref{est-definition-trace-of-renormalized-defect-measure}, namely $\gamma_+ f := f_\gamma$.

On the other hand, multiplying the first equation in \eqref{boltzmann-approximate-lions} by a positive test function $\phi \in \mathcal{D}([0,T]\times\bar{\Omega}\times\mathbb{R}^3)$,
 	\begin{align}
 	\label{ff}
 	\begin{split}
&\int_{0}^{T} \int_{\mathcal{O}} \big( \beta(f^n)(\partial_t \phi  + v \cdot \nabla_x \phi) + Q^n(f^n, f^n)\beta'(f^n) \phi\big) \mathrm{d}x \mathrm{d}v \mathrm{d}t \\
&=   \int_{0}^{T}\int_{\Sigma_+} \beta(\gamma_+ f^n)(s) \phi \cdot {|\mathrm{n}(x) \cdot v|} \bd v \bd \sigma_x \bd s   - \int_{0}^{T}\int_{\Sigma_-} \beta(g)(s)\phi \cdot |\mathrm{n}(x) \cdot v| \bd v \bd \sigma_x \bd s\\
& +  \int_{\mathcal{O}} \beta(f^n)(T)(\tau, \cdot) \phi \mathrm{d}x \mathrm{d}v   -  \int_{\mathcal{O}} \beta(f^n)(0)(\tau, \cdot) \phi \mathrm{d}x \mathrm{d}v.
\end{split}
\end{align}
As $n$ goes to infinity,   there exists a uniform lower bound to the left hand of \eqref{ff}. Indeed, by Theorem \ref{theorem-interior},
\begin{align}
\label{ff0}
\begin{split}
& \int_{0}^{T} \int_{\mathcal{O}} \big( \beta(f)(\partial_t \phi  + v \cdot \nabla_x \phi) + Q(f, f)\beta'(f) \phi\big) \mathrm{d}x \mathrm{d}v \mathrm{d}t \\
&  \le \liminf_{n\to \infty}\int_{0}^{T} \int_{\mathcal{O}} \big( \beta(f^n)(\partial_t \phi  + v \cdot \nabla_x \phi) + Q^n(f^n, f^n)\beta'(f^n) \phi\big) \mathrm{d}x \mathrm{d}v \mathrm{d}t.
\end{split}
\end{align}

At the same time, we can get a uniform upper bound for the right hand of \eqref{ff}. In details, since $(f^n)$ is a strong convergence sequence in $L^\infty([0,T];L^1(\mathcal{O}))$,
\begin{align}
\label{ff1}
\begin{split}
&\limsup_{n\to \infty}\bigg( \int_{0}^{T}\int_{\Sigma_+} \beta(\gamma_+ f^n)(s) \phi \cdot {|\mathrm{n}(x) \cdot v|} \bd v \bd \sigma_x \bd s   - \int_{0}^{T}\int_{\Sigma_-} \beta(g)(s)\phi \cdot |\mathrm{n}(x) \cdot v| \bd v \bd \sigma_x \bd s\\
& +  \int_{\mathcal{O}} \beta(f^n)(T)(\tau, \cdot) \phi \mathrm{d}x \mathrm{d}v   -  \int_{\mathcal{O}} \beta(f^n)(0)(\tau, \cdot) \phi \mathrm{d}x \mathrm{d}v\bigg)\\
&  \le \bigg( \limsup_{n\to \infty} \int_{0}^{T}\int_{\Sigma_+} \beta(\gamma_+ f^n)(s) \phi \cdot {|\mathrm{n}(x) \cdot v|} \bd v \bd \sigma_x \bd s   - \int_{0}^{T}\int_{\Sigma_-} \beta(g)(s)\phi \cdot |\mathrm{n}(x) \cdot v| \bd v \bd \sigma_x \bd s\\
& + \limsup_{n\to \infty} \int_{\mathcal{O}} \beta(f^n)(T)(\tau, \cdot) \phi \mathrm{d}x \mathrm{d}v   - \liminf_{n\to \infty} \int_{\mathcal{O}} \beta(f^n)(0)(\tau, \cdot) \phi \mathrm{d}x \mathrm{d}v\bigg)\\
&  \le \bigg( \limsup_{n\to \infty} \int_{0}^{T}\int_{\Sigma_+} \beta(\gamma_+ f^n)(s) \phi \cdot {|\mathrm{n}(x) \cdot v|} \bd v \bd \sigma_x \bd s   - \int_{0}^{T}\int_{\Sigma_-} \beta(g)(s)\phi \cdot |\mathrm{n}(x) \cdot v| \bd v \bd \sigma_x \bd s\\
& +  \int_{\mathcal{O}} \beta(f)(T)(\tau, \cdot) \phi \mathrm{d}x \mathrm{d}v   -  \int_{\mathcal{O}} \beta(f_0)(\tau, \cdot) \phi \mathrm{d}x \mathrm{d}v\bigg).
\end{split}
\end{align}
We claim that for concave function $\beta$ and non-negative test function $\phi$, if $ \gamma_+ f^n \rightharpoonup   f_\gamma$ in $L^1((0,T)\times\Sigma_+; {|\mathrm{n}(x) \cdot v|} \bd v \bd \sigma_x \bd s)$, then
\begin{align}
\label{ff2}
\limsup_{n\to \infty} \int_{0}^{T}\int_{\Sigma_+} \beta(\gamma_+ f^n)(s) \phi \cdot {|\mathrm{n}(x) \cdot v|} \bd v \bd \sigma_x \bd s \le   \int_{0}^{T}\int_{\Sigma_+} \beta(f_\gamma)(s) \phi \cdot {|\mathrm{n}(x) \cdot v|} \bd v \bd \sigma_x \bd s.
\end{align}
Then by (\ref{ff}-\ref{ff2}), we finally verify that $f_\gamma$ satisfies \eqref{est-definition-trace-of-renormalized-defect-measure}, namely
  	\begin{align*}
 \begin{split}
 &\int_{0}^{T} \int_{\mathcal{O}} \big( \beta(f)(\partial_t \phi  + v \cdot \nabla_x \phi) + Q(f, f)\beta'(f) \phi\big) \mathrm{d}x \mathrm{d}v \mathrm{d}t \\
 & \le   \int_{0}^{T}\int_{\Sigma_+} \beta(\gamma_+ f)(s) \phi \cdot {|\mathrm{n}(x) \cdot v|} \bd v \bd \sigma_x \bd s   - \int_{0}^{T}\int_{\Sigma_-} \beta(g)(s)\phi \cdot |\mathrm{n}(x) \cdot v| \bd v \bd \sigma_x \bd s\\
 & +  \int_{\mathcal{O}} \beta(f)(T)(\tau, \cdot) \phi \mathrm{d}x \mathrm{d}v   -  \int_{\mathcal{O}} \beta(f_0)(\tau, \cdot) \phi \mathrm{d}x \mathrm{d}v.
 \end{split}
 \end{align*}
 \begin{proof}[Proof of \eqref{ff2}]
 	Recalling that $ \gamma_+ f^n \rightharpoonup  f_\gamma$,  for each positive $\phi$ with compact support, if we set $ \phi \cdot {|\mathrm{n}(x) \cdot v|} \bd v \bd \sigma_x \bd s $ as a new measure denoted by $\bd_\phi$, we can infer that up to a subsequence
 	\[ \gamma_+ f^n \rightharpoonup   f_\gamma, ~\text{in}~ L^1((0,T)\times\Sigma_+;\bd_\phi). \]
 	Secondly, noticing that $\beta$ is a concave function, by the lower upper semi-continuity of concave function with respect to weak convergence,
 	\[ \limsup_{n}\int_{0}^{T}\int_{\Sigma_+} \beta(\gamma_+ f^n) \bd_\phi \le \int_{0}^{T}\int_{\Sigma_+} \beta(f_\gamma) \bd_\phi. \]
 \end{proof}
So we can choose $\gamma_+ f := f_\gamma$. Then \eqref{est-lions-density-0} becomes
	\begin{align}
	\label{est-lions-density}
	\begin{split}
	& \int_{\mathcal{O}} f \mathrm{d}x \mathrm{d}v + \int_{0}^{t}\int_{\Sigma_+}   \gamma_+ f \xmus \\
	& =    \int_{\mathcal{O}}  f_0 \mathrm{d}x \mathrm{d}v + \int_{0}^{t}\int_{\Sigma_-}   g \xmus , ~ t > 0.
	\end{split}
	\end{align}
As for the energy inequality, recalling that
\[ \gamma_+ f^n \rightharpoonup \gamma_+ f, ~\text{in},~ L^1((0,T)\times\Sigma_+;\xmus),  \]
then  for any fixed $m\in \mathbb{N}^+$, on the characteristic function of ball in $\mathbb{R}^3$ with radius $m$,  $\{v | |v| \le m  \}$ denoted by $\mathbf{1}_m$, we can infer that
\[ |v|^2 \mathbf{1}_{m}\gamma_+ f^n \rightharpoonup  |v|^2 \mathbf{1}_{m} \gamma_+ f, ~\text{in},~ L^1((0,T)\times\Sigma_+;\xmus),  \]
and
\[ |v|^2 \mathbf{1}_{m}  f^n \rightharpoonup  |v|^2 \mathbf{1}_{m}   f, ~\text{in},~ L^\infty((0,T); L^1(\mathcal{O}; \bd v \bd x)).  \]
By the lower semi-continuity of norm,
\begin{align*}
\begin{split}
& \sup_{0\le s \le t}\int_{\mathcal{O}} \mathbf{1}_{m} |v|^2f(s) \mathrm{d}x \mathrm{d}v + \int_{0}^{t}\int_{\Sigma_+} \mathbf{1}_{m}  |v|^2 \gamma_+ f |\mathrm{n}(x)\cdot v| \mathrm{d}v \mathrm{d}\sigma_x\mathrm{d}s \\
& \le    \int_{\mathcal{O}} |v|^2f_0 \mathrm{d}x \mathrm{d}v + \int_{0}^{t}\int_{\Sigma_-} |v|^2 g, \qquad t > 0.
\end{split}
\end{align*}
Taking $m$ to infinity, by Fatou lemma,  we deduce
		\begin{align}
		\label{est-lions-energy}
		\begin{split}
		& \sup_{0\le s \le t}\int_{\mathcal{O}} |v|^2f(s) \mathrm{d}x \mathrm{d}v + \int_{0}^{t}\int_{\Sigma_+}  |v|^2 \gamma_+ f |\mathrm{n}(x)\cdot v| \mathrm{d}v \mathrm{d}\sigma_x\mathrm{d}s \\
		& \le    \int_{\mathcal{O}} |v|^2f_0 \mathrm{d}x \mathrm{d}v + \int_{0}^{t}\int_{\Sigma_-} |v|^2 g \xmus, \qquad t > 0.
		\end{split}
		\end{align}
For the relative entropy inequality, noticing that $h(z)$ is a positive convex function, by the lower semi-continuity of convex functions, we deduce that
	\begin{align}
	\label{est-lions-entropy}
\begin{split}
& H(f|\mathcal{M}) + \int_{0}^{t}\int_{\Sigma_+}h(\gamma_+ f_{\mathcal{M} } ) |\mathrm{n}(x)\cdot v| \mathrm{d}v \mathrm{d}\sigma_x\mathrm{d}s +  \int_{0}^t \mathcal{D}(f)(s) \mathrm{d}x \mathrm{d}v \mathrm{d}s \\
&\le  H(f_0|\mathcal{M})  + \int_{0}^{t}\int_{\Sigma_-}h(g|{\mathcal{M} } ) |\mathrm{n}(x)\cdot v| \mathrm{d}v \mathrm{d}\sigma_x\mathrm{d}s,\qquad t \ge .
\end{split}
	\end{align}
	
\subsection{Local conservation law}
In this subsection, we focus on the  local conservation law: local conservation law and global conservation law.
\subsubsection{Local mass conservation law}
 Choosing function $\phi_1 \in \mathcal{D}([0,T]\times\bar{\Omega})$, multiplying the first equation \ref{boltzmann-approximate-lions} by $\phi_1$, integrating by parts, then we have
\begin{align*}
\begin{split}
& \int_\mathcal{O} f^n(0) \phi_1(0) dv dx - \int_\mathcal{O} f^n(T) \phi_1(T) dv dx  - \int_{0}^{T} \int_\mathcal{O} f^n(s) \partial_t \phi_1(s) dv dx ds\\
& = \int_{0}^{T} \int_\mathcal{O} f^n(s) v \cdot \nabla \phi_1(s) dv dx ds - \int_{0}^{T} \int_\Sigma \gamma f^n(s)   \phi_1(s) v \cdot \mathrm{n} dv d\sigma_x ds
\end{split}
\end{align*}
As
\[ (1+ |v|) f^n \to (1 + |v|) f,~~\text{in}~~ L^1\big((0,T) \times \Omega; dv dx ds\big), \]
and
\[ \gamma_\pm f^n \to \gamma_\pm f,~~\text{in}~~ L^1\big((0,T) \times \Sigma_\pm; d \mu dx ds\big), \]
taking $n\to \infty$, we have
\begin{align*}
\begin{split}
& \int_\mathcal{O} f(T) \phi_1(T) dv dx - \int_\mathcal{O} f(T) \phi_1(T) dv dx  - \int_{0}^{T} \int_\mathcal{O} f(s) \partial_t \phi_1(s) dv dx ds\\
& = \int_{0}^{T} \int_\mathcal{O} f(s) v \cdot \nabla \phi_1(s) dv dx ds - \int_{0}^{T} \int_\Sigma \gamma f(s)   \phi_1(s) v \cdot \mathrm{n} dv d\sigma_x ds.
\end{split}
\end{align*}
Noticing that $\phi_1$ is independent of $v$, then it can be rewritten as
\begin{align*}
\begin{split}
& \int_\Omega \phi_1 \cdot \big( \int_{\mathbb{R}^3}f dv \big)(T)  dx - \int_\Omega \phi_1 \cdot \big( \int_{\mathbb{R}^3}f dv \big)(0)  dx  \\
& - \int_{0}^{T} \int_{\partial\Omega} \partial_t \phi_1(s) \cdot \big( \int_{\mathbb{R}^3} fdv \big)(s) dx ds\\
& = \int_{0}^{T} \int_{\Omega} \nabla \phi_1(s) \cdot  \big( \int_{\mathbb{R}^3} f v dv\big)(s) dx ds - \int_{0}^{T} \int_{\partial\Omega}  \mathrm{n} \cdot \big(\int_{\mathbb{R}^3} \gamma fvdv \big)(s) \phi_1 d\sigma_x ds.
\end{split}
\end{align*}
If $\phi_1(s)|_{\partial\Omega} = 0$ for any $ 0 \le s \le T$, we can conclude that the following holds in the distribution sense
\begin{align*}
\partial_t \int_{\mathbb{R}^3} f^n(t)\mathrm{d}v + \nabla \cdot \int_{\mathbb{R}^3} v f^n  \mathrm{d}v =0.
\end{align*}
For the general case, by the Green formula, the trace of $\int_{\mathbb{R}^3} v f^n(t)dv$	
\begin{align*}
\begin{split}
\mathrm{n} \cdot \gamma\big( \int_{\mathbb{R}^3} v f^n(t)dv \big) &  = \mathrm{n} \cdot \big(\int_{\mathbb{R}^3} \gamma fv \mathrm{d}v \big)(t)\\
& =  \big(\int_{\mathbb{R}^3} \gamma f v \cdot  \mathrm{n}  \mathrm{d}v \big)(t)\\
& = \big(\int_{\mathbb{R}^3} \gamma_+ f \mathrm{d} \mu  \big)(t) - \big(\int_{\mathbb{R}^3} g \mathrm{d} \mu  \big)(t)
\end{split}
\end{align*}
\subsubsection{Local momentum conservation law}Different with the local conservation law, we need to add some defect measure to deduce the local momentum conservation law. Similarly, multiplying the fist equation of \eqref{boltzmann-approximate-lions} by $ v \phi_1$ with $\phi_1 \in \mathcal{D}((0,T)\times\Omega )$, we have
\begin{align*}
\begin{split}
  \int_{0}^{T} \int_\Omega \partial_t \phi_1(s) \int_{\mathbb{R}^3} v f^n(s)  \bd v \bd x \bd s +  \int_{0}^{T} \int_\Omega \nabla \phi_1(s) \int_{\mathbb{R}^3} f^n(s) v\otimes v \bd v \bd x \bd s = 0
\end{split}
\end{align*}
Recalling that \[ v f^n \rightharpoonup v f,~ \text{in},~ L^\infty((0,T); L^1(\mathcal{O}));~~ f^n \to f, a.e,  \]
by Vitalli convergence theorem, we can deduce
\[ v f^n \to v f,~ \text{in}, ~ L^\infty((0,T); L^1(\mathcal{O})). \]

Thus,
\[  \int_{0}^{T} \int_\Omega \partial_t \phi_1(s) \int_{\mathbb{R}^3} v f^n(s)  \bd v \bd x \bd s  \to  \int_{0}^{T} \int_\Omega \partial_t \phi_1(s) \int_{\mathbb{R}^3} v f(s)  \bd v \bd x \bd s.  \]

The only things at our disposal are \eqref{conservations}
and
\[ f^n \to f, \text{in},~ L^1((0,T)\times\mathcal{O}). \]
With these estimates, we can only prove that there exist  distribution-value matrix $\mathrm{M}$ with $\mathrm{M}_{i,j}( i, j =1, 2, 3) \in   D'(0,T)\times\Omega$  such that while $n \to \infty$
\begin{align*}
\int_{0}^{T} \int_\mathcal{O} f^n v\otimes v \cdot \nabla \phi_1 \bd v \bd x \bd s.   \to  \int_{0}^{T} \int_\mathcal{O}   (f v\otimes v ) \cdot \nabla \phi_1 \bd v \bd x \bd s. + <\mathrm{M}, \nabla \phi_1 >.
\end{align*}

All together, we conclude the local conservation law of momentum.

\end{document}